\documentclass{amsart}
\usepackage{lmodern}

\usepackage[latin9]{inputenc}
\usepackage{amsthm}
\usepackage{amssymb}
\usepackage[unicode=true,pdfusetitle,
 bookmarks=true,bookmarksnumbered=false,bookmarksopen=false,
 breaklinks=false,pdfborder={0 0 1},backref=section,colorlinks=false]
 {hyperref}
\hypersetup{
 bookmarks,colorlinks,pagebackref}

\makeatletter
\numberwithin{equation}{section}
\numberwithin{figure}{section}

\usepackage{bbm}\usepackage{tikz}

\usepackage{times}\usepackage{amsfonts}\usepackage{amscd}\usepackage{amsxtra}\usepackage{latexsym}\usepackage{multirow}\usepackage{amsthm}\usepackage{mathtools}\usepackage{cjhebrew}\input{xy}
\xyoption{all}

    \newcommand{\sbt}{\,\begin{picture}(-1,1)(-1,-3)\circle*{2}\end{picture}\ }
\newcommand{\sarc}{\mathrel{\ooalign{$\nabla$\cr
  \hidewidth\raise.3ex\hbox{$\sbt\mkern5mu$}\cr}}}

\newcommand{\nn}{\,\begin{picture}(-1,1)(-1,-3)\scalebox{.5}{n}\end{picture}\ }
\newcommand{\narc}{\mathrel{\ooalign{$\nabla$\cr
  \hidewidth\raise.05ex\hbox{$\nn\mkern7mu$}\cr}}}

\theoremstyle{plain}

\newtheorem*{conjectuur*}{Conjecture}

\swapnumbers
\newtheorem{theorem}[subsection]{Theorem}

\newtheorem{lemma}[subsection]{Lemma}
\newtheorem{proposition}[subsection]{Proposition}

\theoremstyle{definition}

\theoremstyle{remark}

\newtheorem{remark}[subsection]{Remark}

\swapnumbers

\newcommand{\emptyprop}{q}

\newcommand{\sX}[1]{\mathcal{X}_{#1}}
\newcommand{\sY}[1]{\mathcal{Y}_{#1}}
\newcommand{\sZ}[1]{\mathcal{Z}_{#1}}
\newcommand{\bA}{\mathbb{A}}

\newcommand{\bX}{\mathbb{X}}

\newcommand{\limsieves}[1]{\categ{LSieves}_{#1}}

\newcommand{\nmes}[2]{\categ{Ms}^{#2}\categ{Sieves}_{#1}}

\newcommand{\set}{\categ{Set}}

\newcommand{\nsieves}[2]{\categ{s}^{#2}\categ{Sieve}_{#1}}
\newcommand{\mot}{\text{\fontsize{16}{26}\selectfont\cjRL{M}}}

\newcommand{\mmot}{\categ{Mes}\text{\fontsize{16}{26}\selectfont\cjRL{M}}}

\newcommand{\sO}[1]{\mathcal{O}_{#1}}

\newcommand{\spec}[1]{\operatorname{Spec}(#1)}

\newcommand{\sS}[3]{\mathcal{S}_{#1}^{#2}{(#3)}}

\newcommand{\sC}[3]{\mathcal{C}_{#1}^{#2}{(#3)}}
\newcommand{\sT}[3]{\mathcal{T}_{#1}^{#2}{(#3)}}

\newcommand{\bsC}[3]{\bar{\mathcal{C}}_{#1}^{#2}{(#3)}}

\newcommand{\into}{\hookrightarrow}

\newcommand{\maxim}{\mathfrak m}
\newcommand{\nat}{\mathbb N}

\newcommand{\op}{\operatorname}

\newcommand{\zet}{\mathbb Z}

\newcommand{\commdiagram}[9][]{%
\begin{equation}
{\newcommand{\tmpprop}{#1q} 
\if\tmpprop\emptyprop \relax\else \label{#1}\fi}
\begin{aligned}%
\mbox{
\begin{picture}(130,90)%
\put(120,70){\vector( 0,-1){50}}%
\put(10,80){\vector( 1, 0){100}}%
\put(0,70){\vector( 0,-1){50}}%
\put(10,10){\vector( 1, 0){100}}%
\put(115,80){\makebox(0,0)[l]{$#4$}}%
\put(5,80){\makebox(0,0)[r]{$#2$}}%
\put(115,10){\makebox(0,0)[l]{$#9$}}%
\put(5,10){\makebox(0,0)[r]{$#7$}}%
\put(-3,50){\makebox(0,0)[r]{$#5$}}
\put(123,50){\makebox(0,0)[l]{$#6$}}
\put(60,3){\makebox(0,0)[c]{$#8$}}
\put(60,88){\makebox(0,0)[c]{$#3$}}
\end{picture}}
\end{aligned}
\end{equation}}

\newcommand{\commtrianglefront}[7][]{%
\begin{equation}
{\newcommand{\tmpprop}{#1q} 
\if\tmpprop\emptyprop \relax\else \label{#1}\fi}
\begin{aligned}%
\mbox{
\begin{picture}(120,80)%
\put(55,68){\vector(-1,-2){30}}
\put(65,68){\vector(1,-2){30}}
\put(30,5){\vector(1,0){60}}
\put(60,75){\makebox(0,0)[c]{$#2$}}
\put(25,5){\makebox(0,0)[r]{$#4$}}
\put(95,5){\makebox(0,0)[l]{$#6$}}
\put(60,0){\makebox(0,0)[c]{$#5$}}
\put(37,43){\makebox(0,0)[r]{$#3$}}
\put(83,43){\makebox(0,0)[l]{$#7$}}
\end{picture}}
\end{aligned}
\end{equation}}

\newcommand{\commtriangleback}[7][]{%
\begin{equation}
{\newcommand{\tmpprop}{#1q}
\if\tmpprop\emptyprop \relax\else \label{#1}\fi}
\begin{aligned}%
\mbox{
\begin{picture}(120,80)%
\put(55,70){\vector(-1,-2){30}}
\put(65,70){\vector(1,-2){30}}
\put(30,5){\vector(1,0){60}}
\put(60,75){\makebox(0,0)[c]{$#2$}}
\put(25,5){\makebox(0,0)[r]{$#6$}}
\put(95,5){\makebox(0,0)[l]{$#4$}}
\put(60,0){\makebox(0,0)[c]{$#5$}}
\put(37,43){\makebox(0,0)[r]{$#7$}}
\put(83,43){\makebox(0,0)[l]{$#3$}}
\end{picture}}
\end{aligned}
\end{equation}}

\usepackage{mathabx}

\newcommand{\fat}{\mathfrak z}

\newcommand{\categ}[1]{\mathbbmss{#1}}

\newcommand{\sieves}[1]{\categ{Sieve}_{#1}}

\newcommand{\fld}{\kappa}

\newcommand{\ord}[2]{\op{ord}_{#1}(#2)}

\newcommand{\grot}[1]{{\mathbf {Gr}(#1)}}

\newcommand{\fatpoints}[1]{\categ {Fat}_{#1}}

\newcommand{\lef}{{\mathbb L}}

\makeatother

\begin{document}

\title{A primitive change of variables formula for schemic motivic integration}

\author{Andrew R. Stout}

\address{Andrew R. Stout\\
 Graduate Center, City University of New York\\
 365 Fifth Avenue\\
10016\\
U.S.A. \& Universit\'e Pierre-et-Marie-Curie \\
 4 place Jussieu \\
75005\\
Paris, France}

\email{astout@gc.cuny.edu}

\maketitle

\begin{abstract}
We develop further the theory of integrable functions within the theory of relative simplicial
motivic measures. We provide a primitive change of variables formula for this theory.
\end{abstract}
\setcounter{tocdepth}{1} \tableofcontents{}

\section{Introduction}

In \cite{me2}, I developed the notion of an $\sigma$-integrable function within the theory of schemic motivic integration. In this paper, we generalize the notion of integrability so that the analogy with constructible motivic integration can be seen. In particular, we prove a schemic version of Theorem 10.1.1 of \cite{CL}. Note that there is no cell-decomposition theorem in schemic motivic integration. So, naturally, there are going to be quite a few results from constructible motivic integration which do not carry over to the schemic theory. Thus, we propose the question of what the analogues of Statements A7 and A8 of Theorem 10.1.1 could be.

In this paper, we observe that simplicial sets are not all that useful when dealing with some rather natural types of cach\'e functions; however, this does not present a significant difficulty for the theory since we  rarely used the face and degeneracy maps in \cite{me3} and since the results of \cite{me2} and \cite{me3} readily generalize to countably indexed families of sieves.

One of the key insights of this paper is that when it comes to dealing with cach\'e functions there is no reason to avoid sheafifying as I did in \cite{me3}. In fact, one sees that once we sheafify, we have two different types of integrable functions: strongly integrable functions and just general integrable functions. The important point is that strongly integrable functions would basically be the analogue of the notion of integrable constructible motivic function in \cite{CL}; however, sheafifying allows us to consider a more general notion of integrability (the notion of an integrable function).

The theory of schemic motivic integration (over finite arcs) was created by Hans Schoutens in \cite{schmot1} and \cite{schmot2}. In \cite{me1} and \cite{me2}, I began to develop the notion of the schemic (geometric) integral which could take place over infinite arcs. In \cite{me3}, I developed the more general theory of schemic motivic integration over infinite arcs in the framework of \S 4, 5, \& 6 of 
\cite{CL}. However, this was just a stepping stone toward the general notion of integrability over infinite arcs as outlined in this paper. Much of the terminology and notation is pulled from  \cite{schmot1}, \cite{schmot2}, \cite{me3}, and \cite{me2}. The reader may consult these previous papers for the necessary background material.

This work was partially supported by the chateaubriand fellowship, Prof. F. Loeser, and DSC research grant.

\section{Explicit description of pushforward and pullback} \label{2}

Let $\sY\ $ be in $\categ{s}^n\limsieves{\fld}$ and consider the element $\sX\ $ of $\categ{s}^n\limsieves{\sY{}}$ defined by $\sX\ = \sY{} \times(\bA_{\fld}^{m})$ which is given by the condition
\begin{equation}
\sX{\sigma} = \sY\sigma\times\bA_{\fld}^{m} \ \ \ \forall \sigma \in \prod_{j=1}^{n}\Delta^{\circ}
\end{equation}
Assume that the projection morphism $p : \sX\ \to \sY\ $ is continuous. Then, it induces the pull-back morphism  $p^{\#}$ which is the standard sheaf theoretic morphism from the sheaf of semirings $\sT{}{+}{-}$ 
on $\sY\ $ to the sheaf of semirings $p_*\sT{}{+}{-}$  on $\sY\ $. This morphism is none other than precomposition of a total function $f$ on $\sY\ $ with the map $p$. Applying the global sections functor, we obtain a homomorphism of semirings
\begin{equation}
\sT{}{+}{\sY\ }\otimes_{\sT{0}{+}{\sY{}}}\sT{0}{+}{\sX\ } \to \sT{}{+}{\sX\ } 
\end{equation}
via sending $f\otimes \mathbbm{1}_{\sZ\ }$ to $p_{\#}(f)\cdot  \mathbbm{1}_{\sZ\ }$. We would like this to be an isomorphism although we will soon see that this is in general not the case at all. Surjectivity is immediate.  To see why injectivity fails in general, consider
a  general element $g$ of the domain. It is of the form
\begin{equation}
g = \sum_j (\sum_i a_{ij}\alpha_{ij}\lef^{\beta_{ij}}\otimes \mathbbm{1}_{\sZ{j}})
\end{equation}
with $a_{ij}$ in $\bA_{+}$, $\alpha_{ij}$ and $\beta_{ij}$ in $\sS{}{+}{\sY\ }$.
If $g$ is sent to zero under the morphism in question, then because we are working in a semiring, we immediately obtain that $\sum_i a_{ij}\alpha_{ij}\lef^{\beta_{ij}}\otimes \mathbbm{1}_{\sZ{j}}$ is sent to zero for each $j$. 
Combining the terms in this sum such that $(\beta_i \circ p)|_{\sZ{j}} = (\beta_k \circ p)|_{\sZ{j}}$ for some $i$ and $k$, we immediately obtain that all the coefficients of this sum (after precomposing with $p$) are zero. This implies that 
$ a_{ij}\alpha_{ij}\circ p$ is zero on $\sZ{j}$ for each $i$ and each $j$. It is irrelevent what the values are outside of $\sZ{j}$ as we are tensoring with  $\sZ{j}$'s characteristic function. Therefore, $ a_{ij}\alpha_{ij}$ is zero on
$p(\sZ{j})$. However, all this shows is that
$a_{ij}\alpha_{ij} \otimes \mathbbm{1}_{p(\sZ{j})}$ is zero. Classically speaking, one could use this to prove that $a_{ij}\alpha_{ij} \otimes \mathbbm{1}_{\sZ{j}}$ is also zero as we could make use of quantifier elimination for the language of presburger sets
to show that these conditions are equivalent.  Since we are not working within a model-theoretic framework in general, we will need to find new approaches to problems like these.

The reason we would like such an isomorphism is that it would help us to construct a ring homomorphism from $\sC{}{}{\sX\ }$ to $\sC{}{}{\sY\ }$ in a way that mirrors the work of \S $5.2, \ 5.3,$ and $5.6$ of \cite{CL}. Our work with the functorial approach has paid off as such a morphism is  induced from $p: \sX\ \to\sY\ $. This is a special case of Theorem 7.2 in \cite{me3}. For the benefit of the reader, we work this out in this specific case:

\subsection{Projection along $\fld$-variables}
Let $\mot_{\sX{}}$ be a motivic site. Then we would like to define a canonical ring homomorphism 
\begin{equation}
\sC{}{}{\sX{}, \mot_{\sX{}}} \to \sC{}{}{\sY{},p(\mot_{\sX{}})} \ ,
\end{equation}
when $\sY\ \in \nsieves{\fld}{n}$ and $\sX\ = \sY\ \times \bA_{\fld}^{m}$.
Note that  $p(\mot_{\sX{}})$ is defined in equation 30 of \cite{me2}, and there it is shown that $p$ actually induces a morphism of motivic sites $\mot_{\sX{}} \to p(\mot_{\sX{}})$. Furthermore, it is proven in Corollary 5.3 of loc. cit. that this induces a ring homomorphism
\begin{equation}
\grot{\mot_{\sX{}}} \to \grot{p(\mot_{\sX{}})} \ .
\end{equation}
However, by the definition of cach\'e functions, we need to first consider the measurable seives coming from $\mot_{\sX{}}$. We claim therefore that $p$ actually induces a morphism from $\mmot_{\sX{}}$ to  $\categ{Mes} p(\mot_{\sX{}})$.
Note that $\sX\ $ is actually a projective limit of $\sY{\maxim}\times\bA_{\fld}^{m}$ where $\maxim$ runs over the fat points in the point system of $\sY{}$ as the point system for $\bA_{\fld}^{\fld}$ is just $\{\spec\fld\}$. More details regarding this product can be found 
in the proof of Theorem 5.5 of \cite{me2}. As the simplicial arc functor $\sarc_{\maxim}$ preserves products, the assumption that $\sX\ $ is measurable is equivalent to $\sY\ $ being measurable. Now, for a general element $\mathcal{S}$ of $\mmot_{\sX{}}$ , it is exactly the same 
to show that $p(\mathcal{S})$ is measurable. So, in fact,  we have an isomorphism of motivic sites
\begin{equation}
\mot_{\sX{}} \cong p(\mot_{\sX{}}) \ .
\end{equation}
Then, by Theorem 5.22 of loc. cit., we also have an isomorphism of rings $\grot{\mot_{\sX{}}} \cong \grot{p(\mot_{\sX{}})}$. 
Thus, tensoring with this isomorphism gives a surjective homomorphism of rings from $\sC{}{}{\sX\ , \mot_{\sX\ }}$ to  $\sC{}{}{\sY\ ,p(\mot_{\sX\ })}$, 
which we will again denote by $p_{\#}.$ Note that in the case of $\mot_{\sX\ } =\categ{s}^n\limsieves{\sX{}}$, we get an isomorphism of motivic sites $\categ{s}^n\limsieves{\sX{}} \to  \categ{s}^n\limsieves{\sY{}}$. Form now on and until \S4, we will suppress the motivic site in $\sC{}{}{\sX{}, \mot_{\sX{}}}$ in this 
case-- i.e, we will write $\sC{}{}{\sX{}}$ for $\sC{}{}{\sX{}, \categ{s}^n\limsieves{\sX{}}}$. So, in particular, we have a surjective ring homomorphism
\begin{equation}
p_{\#}:\sC{}{}{\sY{}\times\mathbb{A}_{k}^{m}} \to \sC{}{}{\sY\ }
\end{equation}
for any $\sY\ \in\categ{s}^n\limsieves{\fld}$. This morphism is just the morphism of global sections of the respective sheaves induced by pushforward under the global sections functor. Tensoring by $\bar\bA$, we also arrive at a surjective ring homomorphism from  $\bsC{}{}{\sX\ , \mot_{\sX\ }}$ to  $\bsC{}{}{\sY\ ,p(\mot_{\sX\ })}$, which can be described in exactly the same way.

\subsection{The projection formula}

In the previous section, we made use of the ''pullback'' of motivic sites to constuct the pushforward $p_{\#}:\sC{}{}{\sY{}\times\mathbb{A}_{k}^{m}} \to \sC{}{}{\sY\ }$ of a projection. We would like to now discuss in detail the concept of pullback of cach\'e functions under any morphism $f : \sX{} \to \sY{}$ in $\categ{s}^n\limsieves{\sX{}}$. First note that $f$ induces a morphism of motivic sites from $\categ{s}^n\limsieves{\sY{}}$ to $\categ{s}^n\limsieves{\sX{}}$ by sending $\mathcal{S}$ to $\mathcal{S}\times_{\sY{}}\sX{}$.  We denote this morphism of motivic sites by $f^{-1}$.
Note that $f^{-1}$ restricts to a morphism $\nmes{\sY{}}{n} \to f^{-1}(\nmes{\sY{}}{n})$ and by using Theorem 7.1 of \cite{me3}, we obtain a morphism $\mmot_{\sY{}}$ to $\categ{Mes}f^{-1}(\mot_{\sX{}})$. This induces a ring homomorphism 
\begin{equation}\label{pub}
 \grot{\mmot_{\sY{}}}\to\grot{\categ{Mes}f^{-1}(\mot_{\sX{}})} \ ,
\end{equation}
which is just induced by the global sections functor applied to the pullback morphism of sheaves discussed in Theorem 7.2 of \cite{me3}. 
We can also define a morphism  $f^{\#}$ from the sheaf of rings of permissible functions on $\sY{}$ to the direct image of the sheaf of rings of permissible functions on $\sX\ $ by simplying precomposing by $f$.
This also induces a morphism  $f^{\#}$ from the sheaf of rings of total functions on $\sY{}$ to the direct image of the sheaf of rings of total functions on $\sX\ $ via 
$f^{\#}(a \alpha \lef^{\beta}) := af^{\#}(\alpha ) \lef^{f^{\#}(\beta)}$ where $a\in\bA$ and $\alpha, \beta\in \sS{}{}{\mathcal{U}}$ for any admissible open $\mathcal{U}$ of $\sY{}$. Tensoring this morphism with the pullback morphism of grothendieck rings defined in \ref{pub}, we obtain the pullback morphism $f^{\#}$ from the presheaf of cach\'e functions on $\sY{}$ with respect to the motivic site $\mot_{\sY{}}$ to the direct image of the presheaf of cach\'e functions on $\sX{}$
with respect to the motivic site $f^{-1}(\mot_{\sY{}})$. In particular, by applying the global sections functor, we obtain the ring homomorphism
\begin{equation}
f^{\#} : \sC{}{}{\sY{},\mot_{\sY{}}} \to \sC{}{}{\sY{},f^{-1}(\mot_{\sY{}})} \ .
\end{equation}
Note that if $\mot_{\sY{}} =\categ{s}^n\limsieves{\sY{}}$, then $f^{-1}(\mot_{\sY{}})\subset\categ{s}^n\limsieves{\sX{}}$, and thus, we end up with the pushforward ring homomorphism from $\sC{}{}{\sY{}}$ to $\sC{}{}{\sX{}}$. Note that we denote the sheafication functor by $(-)^a$, but we will not have much need for it until \S 4. 

Now, when we restrict ourselves to the case where $p : \sY{} \times \bA_{\fld}^{m} \to \sY{}$ is the projection morphism, we quickly obtain the following projection formula
\begin{equation}
p_{\#}(xp^{\#}(y)) = p_{\#}(x)y
\end{equation}
where $x \in \sC{}{}{ \sY{} \times \bA_{\fld}^{m}}$ and  $y \in \sC{}{}{\sY{}}$. This is just by applying the fact that the pullback functor is the right adjoint of the pushforward functor.

We would like to now discuss in detail the concept of pushforward of cach\'e functions under any morphism $f : \sX{} \to \sY{}$ in $\categ{s}^n\limsieves{\fld}$. First note that $f$ induces a morphism of motivic sites from $\categ{s}^n\limsieves{\sX{}}$ to $\categ{s}^n\limsieves{\sY{}}$ by sending $\mathcal{S} \subset \sX{}\times\bA_{\fld}^{m}$ to $(f\times\mbox{id}_{\bA_{\fld}^{m}})(\mathcal{S})$. We again denote this morphism of motivic sites by $f$.
Note that $f$ restricts to a morphism $\categ{Mes}\categ{s}^n\limsieves{\sX{}} \to f\categ{Mes}\categ{s}^n\limsieves{\sX{}})$ and by using Theorem 7.1 of \cite{me3}, we obtain a morphism $\mmot_{\sX{}}$ to $\categ{mes}f(\mot_{\sX{}})$. This induces a ring homomorphism 
\begin{equation}\label{pib}
 \grot{\mmot_{\sX{}}}\to\grot{\categ{Mes}f(\mot_{\sX{}})} \ ,
\end{equation}
which is just induced by the global sections functor applied to the pushforward morphism of sheaves discussed in Theorem 7.2 of \cite{me3}.  We have already shown that there is a pushforward morphism $f_{\#}$ from sheaf of rings $\sT{}{}{-}$ 
on $\sX\ $ to the sheaf of rings $f^{-1}\sT{}{}{-}$  on $\sX\ $ in \cite{me3} (not just for projections as in the beginning of \S 2). Thus, by tensoring with the ring homomorphism defined in \ref{pib}, we 
obtain the pushforward morphism $f_{\#}$ from the presheaf of cach\'e functions on $\sX{}$ with respect to the motivic site $\mot_{\sX{}}$ to the direct image of the presheaf of cach\'e functions on $\sY{}$
with respect to the motivic site $f(\mot_{\sX{}})$. In particular, by applying the global sections functor, we obtain the ring homomorphism
\begin{equation}
f_{\#} : \sC{}{}{\sX{},\mot_{\sX{}}} \to \sC{}{}{\sY{},f(\mot_{\sX{}})} \ .
\end{equation}
Note that if $\mot_{\sY{}} =\categ{s}^n\limsieves{\sX{}}$, then $f(\categ{s}^n\limsieves{\sX{}})\subset \categ{s}^n\limsieves{\sY{}}$, and thus, we end up with the  ring homomorphism from $\sC{}{}{\sX{}}$ to $\sC{}{}{\sY{}}$ induced by the applying the global sections functor to the pushforward. By applying the fact that the pullback functor is 
 the right adjoint of the pushforward functor, we obtain the following:
\begin{theorem}\label{projform}
Let $f : \sX\ \to \sY\ $ be any continuous morphism in $\categ{s}^n\limsieves{\fld}$ and let $\mot_{\sX{}}$ be any motivic site relative to $\sX\ $. Then, for any  $x \in \sC{}{}{\sX\ ,\mot_{\sX{}}}$ and  any $y \in \sC{}{}{\sY{},f(\mot_{\sX{}})}$, we have
\begin{equation}
f_{\#}(xf^{\#}(y)) = f_{\#}(x)y \ .
\end{equation}
More generally, if we set $\mathcal{F} =  \sC{}{}{- ,\mot_{\sX{}}|_{-}}^a$ and 
$\mathcal{G} =  \sC{}{}{- ,p(\mot_{\sX{}})|_{-}}^a$, then we have the following isomorphism of sheaves
\begin{equation}
f_{!}(\mathcal{F}\otimes f^{*}\mathcal{G})\cong f_{!}\mathcal{F}\otimes\mathcal{G}
\end{equation}
where $f_{!}$ is the pushforward with compact support functor.
\end{theorem}
Note that all of these results hold when replacing $\sC{}{}{\sX\ ,\mot_{\sX{}}}$ and $\sC{}{}{\sY{},f(\mot_{\sX{}})}$ by $\bsC{}{}{\sX\ ,\mot_{\sX{}}}$ and $\bsC{}{}{\sY{},f(\mot_{\sX{}})}$, respectively.

\section{Schemic Jacobians and integration over families of sieves}

When dealing with $\fld$-varieties as functors of points, the category of fields containing $\fld$ is enough to characterize these functors. In short, $X^{\circ}$ and $Y^{\circ}$ are naturally isomorphic as functors from $\mbox{Fields}_{\fld}$ to $\set$ if and only if they are isomorphic varieties. However, when $X$ and $Y$ are merely separated $\fld$-schemes of finite type, then we have to consider them as functors from $\fatpoints\fld$ to $\set$. Formally, we have the following:

\begin{theorem}
Let $X$ and $Y$ be closed subschemes contained in a separated $\fld$-scheme $Z$ of finite type. Then, $X$ and $Y$ are non-isomorphic over $\fld$ if and only if there exists $\maxim\in\fatpoints\fld$ such that  $X^{\circ}(\maxim)$ and $Y^{\circ}(\maxim)$
are distinct subsets of $Z^{\circ}(\maxim)$.
\end{theorem}

\begin{proof}
This is a restatement of Lemma 2.2 of \cite{schmot1}. A proof can be found there.
\end{proof}

At one level, it is quite nice to have such a characterization of these schemes as it permits us to form the grothendieck ring of a motivic site. However, at another level, we immediately have the need for results to hold relative to any point system. Usually, this can be done without much trouble but not always. For example, when we have need for a notion of order, as in the case of using the jacobianin, we run into difficulties. There are multiple ways one might approach this problem. We have tried one approach to this in \cite{me1} by using reduction maps to classical geometric motivic integration which worked nicely, yet this served in loc. cit. as a bottle-neck for gaining results regarding the motivic integral relative to any point system whatsoever.

For the moment, we will work with the category $\sieves\fld$.  Let $f: \sX{} \to \sY{}$ in $\sieves\fld$ be any rational morphism. We want the jacobian of $f$ to be a permissible function on $\sX{}$. Note that $f$ being rational means that it is the restriction of a morphism of schemes $f: X \to Y$ where $X$ and $Y$ are ambient spaces of $\sX{}$ and $\sY{}$, respectively. Locally, let $X = \spec A$ and $Y=\spec B$. By shrinking $X$ if necessary, we can assume that $A$ is a finitely generated $B$-algebra -- i.e., of the form
\begin{equation}
A=B[x_1\ldots,x_m]/(f_1,\ldots,f_s) \ .
\end{equation}
Let $B' = B[x_1\ldots,x_m]$ and $I =(f_1,\ldots,f_s)$. We have the homomorphism $\delta : I/I^2 \to A\otimes_{B'} \Omega_{B'/\fld}$ which is defined $a\otimes r$ to $a\otimes dr$ where $a\in A$ and $r \in I\cdot B$. Note that
\begin{equation}
 A\otimes_{B'} \Omega_{B'/\fld} \cong \oplus_{i=1}^{m}Adx_i
\end{equation}
and that $I/I^2$ is some finitely generated free $A$-module\footnote{This is because all schemes in this paper are Noetherian so that finitely generated and finite presentation are equivalent. Therefore, $I$ being a finitely presented $B$-algebra implies that $I/I^2$ is a finitely presented $A$-algebra.}  of rank $n$. Therefore, $\delta$ is the unique $A$-linear map defined by $f_i \mapsto df_i.$ This is none other than the Jacobian matrix $\mbox{Jac}(f): A^n \to A^m$ of the morphism $f$. Using the conormal exact sequence, one can show that the cokernel of $\mbox{Jac}(f)$ (i.e., of $\delta$) is $\Omega_{A/B}$.

Now, assume that $f$ is \'etale so that we can form the determinant of $\mbox{Jac}(f)$. This induces a function from $\beta_f: \sX{} \to \nat$ in the following way. Let $\maxim \in \fatpoints\fld$ and let $x : \maxim \to X$ be an element of $\sX{}(\maxim)$. We define 
\begin{equation}
\beta_f(\maxim)(x):= \ell(\mathcal{O}_{\maxim}(\maxim)/|\mbox{Jac}(f)|\circ x) \ .
\end{equation}

The author tried to prove that $\beta_{f}$ is a permissible function; however, there does not seem to be canonical face and degeneracy maps which satisfy the simplicial identities. Upon some reflection, this turns out not to be so important for the general theory as we worked so far with simplicial categories in \cite{me2} and \cite{me3} for motivational reasons (i.e., with the prospect of working over derived stacks). Everything works just as well if we forget the face and degeneracy morphisms in $\Delta^{\circ}$. We will denote this category by $\nat$, and we will replace the word simplicial with the word indexed. In other words, we are now working with families of sieves indexed by $\nat$. All the definitions and theorems that do not involve face and degeneracy morphisms carry over verbatim to this context. For example, it is immediate that $\beta_f \in\sS{}{}{\sX\ }$ whenever $f: \sX\ \to \sY\ $ is a \'etale rational morphism of sieves. For more on the category of indexed sieves see \S $7.3$ of \cite{me2} and \S $7.5$ of \cite{me3}.

Another function involving the jacobian can be obtained by considering a Jordan Holder composition series $\mathcal{A}=\{\mathfrak{a}_i\}$  for $R=\sO{\maxim}(\maxim)$ with $\maxim\in\fatpoints\fld$ as outlined in Remark 4.2 of \cite{schmot1}. There one needs a monomial ordering $<$ to construct $\mathcal{A}$. Then, the height of an element $r$ of $R$ is the minimal $i$ such that it is an element of the ideal $\mathfrak{a}_i$. We will denote the height as $h_{<}(r)$.

\begin{proposition}
Let $R$ be a local artinian ring which contains $\fld$ and which is equipped with a monomial ordering $<$. The height function $h_{<}: R \to \zet/\ell(R)\zet$ is uniquely determined by the pair $(R,<)$ .
\end{proposition}

\begin{proof}
This follows immediately from the construction in Remark 4.2 of loc. cit.
\end{proof}

We also have the order function on the local aritinian ring $R = k[t_1,\ldots,t_s]/I$ which we denote by $\mbox{ord} : R \to \nat$. This is defined as follows. Consider the morphism $i: R \to k[t_i]/(t_{i}^{N})$ where $N$ is the nilpotency of the variable $t_i$. Given any element $g \in  k[t_i]/(t_{i}^{N})$ and any element $g' \in k[[t_i]]$ whose residue class is $g$, we define $\ord{}g = \ord{}{g'}$ where $\ord{}-$ is the valuation on the discrete valuation ring $k[[t_i]]$. Then, for any $r \in R$, we define 
$\ord{}{r} = \sum_i \ord{}{i(r)}$. In other words, $\ord{}{r}$ is the sum of the minimal exponents with respect to the $i$-th coordinate where the sum runs over $i=1,\ldots,s$. We define 
\begin{equation}
\mbox{ordjac}(f)(\maxim)(x) = \ord{}{|\mbox{Jac}(f)|\circ x}
\end{equation}

Finally, we have one last function of interest which can be applied to the jacobian. Given an element $r$ of $R$, we define its degree, denoted by $\mbox{deg}(r)$, as the maximal $d\leq n$ such that $r \in M^d$ (here, we set $M^0 = (1)$) where $M$ is the maximal ideal of $R$ and where $n$ is the nilpotency of $R$. 
We define $\mbox{degjac}(f)$ to be the composition of $ |\mbox{Jac}(f)|\circ x$ and $\mbox{deg}.$ 

Potentially, any of these functions could be interesting to study within the framework of motivic integration. Naively, the length function seems to be the most natural in this context. However, $\mbox{ordjac}$ is closest thing to a generalization of the concept used in \cite{CL}. Also, since we have already made use of this function in \cite{me1} to prove a simple change of variables formula for schemic geometric motivic integration, we will continue to mostly focus on this function.

\begin{proposition} (The chain rule of $\mbox{ordjac}$):
Let $g : \sX{} \to \sY{}$ and $f :\sY{} \to \sZ{}$ be two \'etale  rational morphisms of sieves. Then, for all $\maxim \in \fatpoints\fld$, 
\begin{equation}
\mbox{ordjac}(f\circ g)(\maxim)(x) = \mbox{ordjac}(g)(\maxim)(f\circ x) + \mbox{ordjac}(f)(\maxim)(x) \ . 
\end{equation}
for all $x \in \sX{}(\maxim)$.
\end{proposition}

\section{The general notion of schemic integration}

Let $\sX \ $ be an element of $\limsieves{\sY{}}$ and let $j: \sX{}\to \sY{}$ be the structure morphism. 
Consider a function $\varphi \in \sT{}{}{\sX{}}$.
 We say that $\varphi$ is  $\sY{}$-integrable
if there exists a cover of $\sY{}$ with admissible opens $\sY{\sigma}$ with $\sigma \in \nat$ such that 
\begin{enumerate}
\item $\sZ{\sigma}=j^{-1}(\sY{\sigma})$ is an admissible open of $\sX{}$ for all $\sigma\in \nat$.
\item for each $\sigma$, there exists an automorphism $\gamma_{\sigma}$  of $\sZ{\sigma}$.
\item for all $\sigma \in \nat$, there exists $\psi_{\sigma} \in \sT{}{}{\sY{\sigma}}$ such that $\varphi|_{\sZ{\sigma}}=\gamma_{\sigma}^{*} j^{*}\psi_{\sigma}$ for all $\sigma \in \nat$.
\item Define $\sY{\bullet}$ to be the indexed sieve given by $\sY{\sigma}$ and let $\psi_{\bullet}$ correspond to 
the total function on $\sY{\bullet}$ determined by the $\psi_{\sigma}$. Then, we require  $\psi_{\bullet}$ to be $\sigma$-integrable over $\sY{\sigma}$ for all $\sigma$.
\end{enumerate}
Additionally, if all of the $\sY{\sigma}$ can be chosen to be mutually disjoint, then we say that $\varphi$ is strongly integrable.

We denote the subring of all $\sY{}$-integrable functions in $\sT{}{}{\sX{}}$ by $I_{\sY{}}\sT{}{}{\sX{}}$ and the subring of all strongly integrable functions $SI_{\sY{}}\sT{}{}{\sX{}}$. Likewise, we define
\begin{equation}
\begin{split}
I_{\sY{}}\sC{}{}{\sX{},\mot_{\sX{}}} &:= \grot{\mmot_{\sX{}}}\otimes_{\sT{}{0}{\sY{}}}  I_{\sY{}}\sT{}{}{\sX{}} \\
SI_{\sY{}}\sC{}{}{\sX{},\mot_{\sX{}}} &:= \grot{\mmot_{\sX{}}}\otimes_{\sT{}{0}{\sY{}}}  SI_{\sY{}}\sT{}{}{\sX{}} \ .
\end{split}
\end{equation}

\begin{theorem} \label{inducedmor}
Let $f: \sX{} \to \sX{}^{\prime}$ be a continuous morphism of  limit $\sY{}$-sieves. Then, the pushforward $f_{\#}$  restricts to a
ring homomorphism 
\begin{equation}
f_{\#}: I_{\sY{}}\sC{}{}{\sX{},\mot_{\sX{}}} \to  I_{\sY{}}\sC{}{}{\sX{}^{\prime},f(\mot_{\sX{}})} \ .
\end{equation} 
\end{theorem}

\begin{proof}
For each real $q> 1$, we have the rings $\sT{}{q}{\sX{}}$ and $\sT{}{q}{\sX{}^{\prime}}$ obtained by applying the evaluation homomorphism $\nu_q$. Moreover, they define sheaves of topological rings. Thus, the pushforward actually induces a continuous ring homomorphism $f_{\#}^q : \sT{}{q}{\sX{}} \to \sT{}{q}{\sX{}^{\prime}}$. The theorem follows immediately from this fact.

\end{proof}

\begin{remark}
The pushforward actually restricts to a morphism of presheaves (or of sheaves if one wants to sheafify) as usual. We write $I\sC{}{}{\sX{},\mot_{\sY{}}}$ for $I_{\sY{}}\sC{}{}{\sX{},\mot_{\sY{}}}$ when $\sY{} = (\spec\fld)^{\circ}$.
\end{remark}

\begin{remark}
Note that for families indexed by $\nat^j$ for some $j> 1$, all of the definitions and results will be the same. Also, tensoring by $\bar\bA$ is enough to define the notion of weakly integrable as in \cite{me2} -- i.e., everything is exactly the same for $I_{\sY{}}\bsC{}{}{\sX{},\mot_{\sX{}}}$. In the interest of brevity, we develop the theory for $\sC{}{}{-}$ over $\limsieves{\sY{}}$ only. Finally, every sieve is trivially a limit sieve, so all of the definitions and results go through with respect to sieves.
\end{remark}

\begin{proposition}\label{clim}
Let $\sX{}$ and $\sY{}$ be limit simplicial sieves with point system $\bX$ where $\fat = \lim\bX$. Assume further that $\sX{}$ is a $\sY{}$-sieve. Then, 
\begin{equation}
I_{\sY{}}\sC{}{}{\sX{},\mot_{\sX{}}} = \injlim I_{\sY{\maxim}}\sC{}{}{\sX{\maxim},\mot_{\sX{\maxim}}} \ .
\end{equation}
Thus, every element $\varphi$ is of the form $\sarc_{\fat}g$ where $g$ is in $ I_{\sY{\maxim}}\sC{}{}{\sX{\maxim},\mot_{\sX{\maxim}}}$ for some $\maxim \in\bX$.
\end{proposition}

We have not yet defined the measure $\mu_{\sY{}}$. For each $\sY{}$-integrable function $\varphi$, we do immediately obtain an element $\psi_{\sigma}$ of 
 $\sC{}{}{\sY{\sigma},\pi_{\sigma}(j(\mot_{\sX{}}))}$  
for all $\sigma \in \nat$ which assures to us that $\varphi$ is integrable. Temporarily, we define a set map $\mu_{\sY{}}^{\sigma} $ by sending $\varphi$ to the function  $\psi_{\sigma}$. 
This defines a ring homomorphism
\begin{equation}
\mu_{\sY{}}^{\sigma} : I_{\sY{}}\sC{}{}{\sX{},\mot_{\sX{}}} \to \sC{}{}{\sY{},j(\mot_{\sX{}})}\ .
\end{equation}

\begin{proposition} Let $\sX{}$ and $\sY{}$ be limit simplicial sieves with point system $\bX$ where $\fat = \lim\bX$. Assume further that $\sX{}$ is a $\sY{}$-sieve. Let $\varphi \in I_{\sY{}}\sC{}{}{\sX{},\mot_{\sX{}}} $. Then, there exists $\maxim \in \bX\ $ and a $g \in  I_{\sY{\maxim}}\sC{}{}{\sX{\maxim},\mot_{\sX{\maxim}}}$ such that 
\begin{equation}
\mu_{\sY{}}^{\sigma}(\varphi) = \sarc_{\fat}\mu_{\sY{}}^{\sigma}(g) \ .
\end{equation}
\end{proposition}
\begin{proof}
This follows immediately from the work of \S 6 of \cite{me3} and from Proposition \ref{clim}.
\end{proof}

\subsection{Sheafificaton and motivic measures}

Philosophically speaking, there is no reason to avoid applying the sheafification functor to the presheaf of total functions as this 
functor is the left adjoint to the forgetful functor. Thus, we will denote by $\sC{\sX{}}{a}{-}$ the sheafification of the 
presheaf which sends an admissible open $\mathcal{U}$ of $\sX{}$ to
 $\sC{}{}{j(\mathcal{U}),\mot_{\sX{}}|_{\mathcal{U}}}$ for 
some fixed motivic site $\mot_{\sX{}}$. Likewise, if $\sX{}$ is a limit $\sY{}$-sieve, then we will denote by $I_{\sY{}}\sC{\sX{}}
{a}{-}$ the subsheaf of
$\sC{\sX{}}{a}{-}$ obtained by sheafifying  the presheaf which sends an admissible open $\mathcal{U}$ of $\sX{}$ to $I_{\sY{}}\sC{}{}{\mathcal{U},\mot_{\sX{}}|_{\mathcal{U}}}$.

Of course, there are many advantages to working with sheaves instead of just presheaves. For example, we have the following canonical isomorphism of rings 
\begin{equation}
\sC{\sY{}}{a}{\sY{}} \cong \prod_{\sigma} \sC{\sY{}}{a}{\sY{\sigma}} 
\end{equation}
when $\sY{}$ decomposes into the disjoint union $\sqcup_i\sY{\sigma}$ for any index $\sigma$.

Now consider a function in $\varphi \in SI_{\sY{}}\sC{\sX{}}
{a}{\sX{}}$. We define $\mu_{\sY{}}(\varphi) \in \sC{\sY{}}
{a}{\sY{}}$ to be the unique element corresponding to $(\mu_{\sY{}}^{\sigma}(\varphi))_{\sigma}$ under the isomorphism above. 
In fact, doing this for each admissible open, we can easily see that we have defined a morphism of sheaves on $\sX\ $
\begin{equation}
\mu_{\sY{}}: SI_{\sY{}}\sC{\sX{}}
{a}{-}\to j^{-1}\sC{\sY{}}
{a}{-}  \ ,
\end{equation}
where $j$ is the structural morphism.
However, the reason we made the distinction between strongly integrable and integrable is to observe that this argument works just as well for any integrable function because of the exact sequence which defines a sheaf. In other words, if $\varphi \in I_{\sY{}}\sC{\sX{}}
{a}{\sX{}}$, then this also uniquely defines an element $\mu_{\sY{}}^{\sigma}(\varphi)$ for each element of the cover $\sY{\sigma}$. It is easy to check that these functions agree on the intersections $\sY{\sigma} \cap \sY{\sigma'}$. Thus, we may glue to obtain a unique function $\mu_{\sY{}}(\varphi) \in \sC{\sY{}}
{a}{\sY{}}$. Doing this for each admissible open, we can easily see that we have defined a morphism of sheaves on $\sX\ $
\begin{equation}
\mu_{\sY{}}: I_{\sY{}}\sC{\sX{}}
{a}{-}\to j^{-1}\sC{\sY{}}
{a}{-}  \ ,
\end{equation}
where $j$ is the structural morphism.

Now from this and Theorem 4.1, it is clear that given a continuous morphism $f: \sX\ \to \sX{}^{\prime}$ of limit $\sY{}$-sieves, we have a commutative diagram 
\[\xymatrix{ & I_{\sY{}}\sC{\sX{}}
{a}{-}  \ar[rd]_{\mu_{\sY{}}}\ar[r]^{f^{\#}} &f^{-1}  I_{\sY{}}\sC{\sX{}^{\prime}}{a}{-}\ar[d]^{\mu_{\sY{}}} \\
& &j^{-1}\sC{\sY{}}{a}{-} } \]
in the topos of sheaves of rings on $\sX{}$.
This commutative diagram can be interpreted as a change of variables formula for schemic motivic integration.

\section{The main theorem}

We present here a schemic analogue to Theorem 10.1.1 of \cite{CL}.

\begin{theorem}
Let $\sY{}$ be an element of $\limsieves\fld$. Every  $\sX{} \in \limsieves{\sY{}}$ can be endowed with a sheaf of rings $I_{\sY{}}\sC{\sX{}}{a}{-}$  with the following properties

\begin{enumerate}
\item Existence of the four functors:
\begin{enumerate}

\item Every continuous morphism $f : \sX{} \to \sX{}^{\prime}$ in  $\sieves{\sY{}}$ induces a morphism of sheaves $f_{\#}$ from 
$I_{\sY{}}\sC{\sX{}}{a}{-}$ to $f^{-1}I_{\sY{}}\sC{\sX{}^{\prime}}{a}{-}$ which is the restriction of the pushforward from 
$\sC{\sX{}}{a}{-}$ to $f^{-1}\sC{\sX{}}{a}{-}$.

\item The analogue of (a) is true for $f^{\#}$, $f_{!}$ and $f^{!}$ (when it exists).
\end{enumerate}
\item Functoriality:
\begin{enumerate}
\item Let $\lambda : \sY{}\to \sY{}^{\prime}$ be a morphism in $\limsieves\fld $. This induces a morphism $\lambda_{+}: \limsieves{\sY{}}\to \limsieves{\sY{}^{\prime}}$ as shown in \cite{me2}. There is a natural inclusion of sheaves $I_{\sY{}^{\prime}}\sC{\lambda(\sX{})}{a}{\lambda_{+}(-)}\subset I_{\sY{}}\sC{\sX{}}{a}{-}$.

\item Given a continuous morphism $f: \sX\ \to \sX{}^{\prime}$ of limit $\sY{}$-sieves, we have a commutative diagram 
\[\xymatrix{ & I_{\sY{}}\sC{\sX{}}
{a}{-}  \ar[rd]_{\mu_{\sY{}}}\ar[r]^{f^{\#}} &f^{-1}  I_{\sY{}}\sC{\sX{}^{\prime}}{a}{-}\ar[d]^{\mu_{\sY{}}} \\
& &j^{-1}\sC{\sY{}}{a}{-} } \]
in the category of sheaves of rings on $\sX{}$.
\end{enumerate}

\item  Integrability:
\begin{enumerate}

\item $I_{\sY{}}\sC{\sX{}}{a}{-}$ is a subsheaf of $\sC{\sX{}}{a}{-}$.

\item $I_{\sY{}}\sC{\sY{}}{a}{-}$ is isomorphic to $\sC{\sY{}}{a}{-}$.
\end{enumerate}

\item Additivity:
If $\sX{} $ is the disjoint union of admissible open subsieves $\sX{i}$ for some index $i$, then the isomorphism 
$\sC{\sX{}}{a}{-} \cong \sC{\sX{i}}{a}{-}$ induces an isomorphism 
 $I_{\sY{}}\sC{\sX{}}{a}{-} \cong \prod_i I_{\sY{}}\sC{\sX{i}}{a}{-}$.

\item Projection Formula:
Given a continuous morphism $f: \sX\ \to \sX{}^{\prime}$ of limit $\sY{}$-sieves, an admissible open $\mathcal{U}$, an element $x \in f^{-1}\sC{\sX{}^{\prime}}{a}{\mathcal{U}}$, and an element $y \in I_{\sY{}}\sC{\sY{}}{a}{\mathcal{U}}$, then 
$xf_{\#}(y)$ is an element of $ f^{-1}I_{\sY{}}\sC{\sX{}^{\prime}}{a}{\mathcal{U}}$ if and only if $f^{\#}(x)y$ is an element of $I_{\sY{}}\sC{\sX{}}{a}{\mathcal{U}}$. If these conditions are satisfied, then $f_{\#}(f^{\#}(x)y) =xf_{\#}(y)$.

\item Inclusions: Let $f : \sX{} \to \sX{}^{\prime}$  be continuous and injective and let $\mathcal{U}$ be an admissible open of 
$\sX{}$. Then, $\varphi$ belongs to $I_{\sY{}}\sC{\sX{}}{a}{\mathcal{U}}$ if and only if $f_{\#}(\varphi)$ belongs to $f^{-1}I_{\sY{}}\sC{\sX{}^{\prime}}{a}{\mathcal{U}}$.

\item Projection along $\fld$-variables:
Consider the morphism $p : \sX{} \to \sX{}^{\prime}$ where $\sX{} = \sX{}^{\prime}\times \bA_{\fld}^{m}$ and $p$ is the projection onto the first factor and $\sX{}^{\prime}$ a limit $\sY{}$-sieve. Then, for any admissible open $\mathcal{U}$  of $\sX{}$ such that $p(\mathcal{U})$ is an admissible open of $\sX{}^{\prime}$,
$\varphi$ belongs to $I_{\sY{}}\sC{\sX{}}{a}{\mathcal{U}}$ if and only if $p_{\#}(\varphi)$ belongs to $I_{\sY{}}\sC{\sX{}^{\prime}}{a}{p(\mathcal{U})}$.

\item Projection along $\nat$-variables:
Consider the indexed family of sieves $\sX{i} = \sY{}$ where $i\in \nat$ and set $\sX{}=\sqcup_i \sX{i}$. For any admissible open $\mathcal{U}$ of the form $\sqcup_i\mathcal{V}$ with $\mathcal{V}$  an admissible open of $\sY{}$, 
$\varphi$ belongs to $I_{\sY{}}\sC{\sX{}}{a}{\mathcal{U}}$ if and only if it is $i$-integrable over $\mathcal{V}$ in the sense of \cite{me3} for all $i$.

\item Strongly integrable functions:
The Statements 1-8 hold for strongly integrable functions as well.
\end{enumerate}

 \end{theorem}

\begin{proof}
Much of what is stated above has already been displayed in this paper. However, we list it in this way to make it clear which statements correspond to which in  Theorem 10.1.1 of \cite{CL}.
\begin{enumerate}
\item \textit{Proof of statement 1}:  We proved part (a) in \S 4.7. Note that as $f_!$ is a subfunctor of $f_{\#}$, the analogue for $f_!$ is immediate. The proof of the analogous statement for the pullback $f^{\#}$ follows along the same lines as the proof for Theorem \ref{inducedmor}.  The statement for $f^{!}$ follows from the adjunction formula.

\item\textit{Proof of statement 2}: We proved part (b) in \S 4.7. For part (a), note that $\lambda_{+}$ is the functor induced by 
inclusion -- i.e., if $\sX{}$ is a $\sY{}$-sieve given by inclusion $i : \sX{} \into \sY{}\times\bA_{\fld}^{m}$, then
 $\lambda_{+}(\sX{}) = \mbox{im}((\lambda\times \mbox{id}_{\bA_{\fld}^{m}})\circ i)$. Note that every admissible open subset of $\lambda_{+}(\sX{})$ is of the form $\lambda_{+}(\mathcal{U})$ where $\mathcal{U}$ is an admissible open subset of $\sX{}$. The result follows.

\item\textit{Proof of statement 3}: Part (a) is immediate and part (b) follows by taking the trivial cover.

\item\textit{Proof of statement 4}: Follows immediate from the exact sequence defining a sheaf.

\item \textit{Proof of statement 5}: We proved the general projection formula in Theorem \ref{projform}. The result follows from statement 1 and statement 2(b).

\item\textit{Proof of statement 6}: One direction is immediate. The other follows from the definitions.

\item\textit{Proof of statement 7}: We proved this in \S 2.

\item\textit{Proof of statement 8}: This follows immediately by definition.

\item\textit{Proof of statement 9}: Immediate.

\end{enumerate}
\end{proof}

\begin{remark}
As we do not have the notion of cell-decomposition (instead choosing to work with general open coverings), it is unclear if there are analogues to statements A7 and A8 of Theorem 10.1.1 of \cite{CL}.
\end{remark}


\begin{thebibliography}{1}

\bibitem{CL} R. Cluckers \& F. Loeser \textit{Constructible motivic functions and motivic integration}, Invent. math. 173, 23-121 (2008)

\bibitem{schmot1} H. Schoutens \textit{Schemic Grothendieck Rings
I}, \url{websupport1.citytech.cuny.edu/faculty/hschoutens/PDF/SchemicGrothendieckRingPart
I.pdf} (2011).

\bibitem{schmot2} H. Schoutens \textit{Schemic Grothendieck Rings
II}, \url{websupport1.citytech.cuny.edu/faculty/hschoutens/PDF/SchemicGrothendieckRingPart
II.pdf} (2011).

\bibitem{me3} A. Stout \textit{Integrable functions within the theory of schemic motivic integration}, \url{http://arxiv.org/abs/1309.1537}

\bibitem{me2} A. Stout \textit{Measurable motivic sites}, \url{http://arxiv.org/abs/1306.4056}

\bibitem{me1} A. Stout \textit{Stability theory for schemes of finite
type and schemic motivic integration}, \url{arxiv.org/abs/1212.1375}


\end{thebibliography}
\end{document}